\documentclass{article}
\usepackage{amsfonts,enumerate,graphics,epsfig,bbm}
\newtheorem{maintheorem}{Theorem}
\newtheorem{theorem}{Theorem}[section]
\newtheorem{lemma}[theorem]{Lemma}
\newtheorem{proposition}[theorem]{Proposition}
\newenvironment{proof}{\noindent {\bf Proof:} }

\newcommand{\qed}{\hfill \ensuremath{\Box}}
\newcommand{\D}{\Delta}
\newcommand{\w}{\omega}
\newcommand{\p}{\mathbb{P}}
\newcommand{\E}{\mathbb{E}}
\newcommand{\ob}{\bar{\omega}}

\begin{document}

\title{Reconstruction threshold for the hardcore model}

\author{Nayantara Bhatnagar, Allan Sly and Prasad Tetali\footnote{Research supported in part by NSF grants DMS-0701043 and-CCR 0910584}}

\maketitle

\begin{abstract}
In this paper we consider the reconstruction problem on the tree for
the hardcore model.  We determine new bounds for the
non-reconstruction regime on the $k$-regular tree showing
non-reconstruction when
\[
\lambda < \frac{(\ln 2-o(1))  \ln^2 k}{2 \ln \ln k}
\]
improving the previous best bound of $\lambda < e-1$.  This is almost tight as reconstruction is known to hold when $\lambda> (e+o(1))\ln^2 k$.  We discuss the relationship for finding large independent sets in sparse random graphs and to the mixing time of Markov chains for sampling independent sets on trees.
\end{abstract}

\section{Introduction}

The reconstruction problem on the tree was originally studied as a
problem in statistical physics but has since found many applications
including in computational phylogenetic reconstruction~\cite{DMR}, the
study of the geometry of the space of random constraint satisfaction
problems~\cite{AchCA,KMRSZ} and the mixing time of Markov
chains~\cite{BCMR,MSW}.  For a Markov model on an infinite tree the
reconstruction problem asks when do the states at level $n$ provide
non-trivial information about the state at the root as $n$ goes to
infinity.    In general the problem involves determining the existence
of solutions of distribution valued equations and as such exact
thresholds are known only in a small number of
examples~\cite{BRZ,EKPS,BCMR,Sly:09}.

In this paper we analyze the reconstruction problem for the hardcore
model on the $k$-regular tree, where each vertex of the tree has
degree $k$.  The hardcore model is a probability distribution over
independent sets $I$ weighted proportionally to $\lambda^{|I|}$.
Previously Brightwell and Winkler~\cite{BW} showed that reconstruction
is possible when $\lambda> (e+o(1))\ln^2 k$.  Improving on their bound
for the non-reconstruction regime, Martin~\cite{Mar} showed that
non-reconstruction holds when $\lambda < e-1$ still leaving a wide gap
between the two thresholds.  Our main result establishes that the bound of
Brightwell and Winkler  is tight up to a $\ln \ln k$
multiplicative factor.

\begin{maintheorem}\label{t:main}
The hardcore model on the $k$-regular tree has non-reconstruction when
\[
\lambda < \frac{(\ln 2-o(1))  \ln^2 k}{2 \ln \ln k}.
\]
\end{maintheorem}

\subsection{The Hardcore Model}
\label{sec:hardcore_model}

For a finite graph $G$ the independent sets $I(G)$ are subsets of the
vertices containing no adjacent vertices.  The hardcore model is a
probability measure over $\sigma\in I(G) \subset \{0,1\}^G$ such that
\begin{equation}\label{e:defnHardcore}
\p(\sigma)=\frac1{Z} \lambda^{\sum_{v\in G}\sigma_v} \mathbbm{1}_{\sigma\in I(G)}
\end{equation}
where $\lambda$ is the \emph{fugacity} parameter and $Z$ is a
normalizing constant.  The definition of the hardcore model can be
extended to infinite graphs by way of the Dobrushin-Lanford-Ruelle
condition which essentially says that for every finite set $A$ the
configuration on $A$ is given by the Gibbs distribution given by a
random boundary generated by the measure outside of $A$.  Such a
measure is called a Gibbs measure and there may be one or
infinitely many such measures (see e.g. \cite{Georgii:88} for more
details).  For every $\lambda$, there exists a unique translation
invariant Gibbs measure on the $k$-regular tree and it is this measure
which we study.

An alternative equivalent formulation of the hardcore model is as a
Markov model on the tree.  An independent set $\sigma$ is generated by
first choosing the root according to the distribution
\[
(\pi_1,\pi_0) = \left(\frac{\w}{1+2\w} \ , \
  \frac{1+\w}{1+2\w}\right)
\]
for some $0<\omega<1$.  The states of the remaining vertices of the graph are generated from their parents' states by taking one step of the Markov transition matrix
$$M = \left(
\begin{array}{cc}
p_{11} & p_{10}\\
p_{01} & p_{00}
\end{array}
\right)
=
\left(
\begin{array}{cc}
0 & 1\\
\frac{\w}{1+\w} & \frac{1}{1+\w}\\
\end{array}
\right).
$$
It can easily be checked that $\pi$ is reversible with respect to $M$
and that this generates a translation invariant Gibbs measure on the
tree with fugacity
\[
\lambda=\omega(1+\omega)^{k-1}.
\]
Restating  Theorem~\ref{t:main} in terms of $\omega$ we have
non-reconstruction when
\begin{equation}\label{e:omegaBound}
\omega \leq \frac1k\Big[ \ln k + \ln
  \ln k - \ln \ln \ln k  - \ln 2 +  \ln \ln 2 - o(1) \Big]=:\ob.
\end{equation}
We will introduce some further notation which we will make use of in the proof.
\begin{eqnarray}
\nonumber \pi_{01} \equiv \frac{\pi_0}{\pi_1} = \frac{1+\w}{\w}, \ \ \ \ \ \
\ \D \equiv \pi_{01} -1 = \frac{1}{\w},
\\ \nonumber \theta \equiv
p_{00} -p_{10} = p_{11}-p_{01} = - \frac{\w}{1+\w}
\end{eqnarray}
A particularly important role is played by $\theta$, the second
eigenvalue of $M$ as is discussed in the following subsection. We
denote by $\p^1_T,\E^1_T$ (and resp. $\p^0_T,\E^0_T$ and
$\p_T,\E_T$) the probability and expectations with respect to the
measure obtained by conditioning on the root $\rho$ of $T$ to be 1 (resp. 0, and
stationary). We let $L=L(n)$ denote the vertices at depth $n$ and
$\sigma(L)=\sigma(L(n))$ denote the configuration on level $n$.  We
will write $\Pr_T[\cdot|\sigma(L)=A]$ to denote the measure
conditioned on the leaves being in state $A\in\{0,1\}^{L(n)}$.

\subsection{The reconstruction problem}

The reconstruction problem on the tree essentially asks if we can
recover information on the root from the spins deep inside the tree.
In particular we say that the model has \emph{non-reconstruction} if
\begin{eqnarray}\label{eq:non-rec}
\Pr_T[\sigma_\rho=1|\sigma(L)]\to \pi_1
\end{eqnarray}
in probability as $n\to\infty$, otherwise the model has
\emph{reconstruction}. Equivalent formulations of non-reconstruction
are that the Gibbs measure is extremal or that the tail
$\sigma$-algebra of the Gibbs measure is trivial~\cite{MP}.  It
follows from Proposition 12 of~\cite{Mossel} that there exists a
$\lambda_R$ such that  reconstruction holds for $\lambda > \lambda_R$
and non-reconstruction holds for $\lambda < \lambda_R$.
The reconstruction problem is to determine the threshold $\lambda_R$.

\subsection{Related Work}

A significant body of work has been devoted to the reconstruction problem on the tree by probabilists, computer scientists and physicists.  The earliest such result is the Kesten-Stigum bound~\cite{KS} which states that reconstruction holds whenever $\theta^2 (k-1) > 1$.  This bound was shown to be tight in the case of the Ising model~\cite{BRZ,EKPS} where it was shown that non-reconstruction holds when $\theta^2 (k-1) \leq 1$.  Similar results were derived for the  Ising model with small external field~\cite{BKMP} and the 3-state Potts model~\cite{Sly:09} which constitute the only models for which exact thresholds are known.  On the other hand, at least when $k$ is large, the Kesten-Stigum bound is known not to be tight for the hardcore model~\cite{BW}.  As such, the most one can reasonably ask to show is the asymptotics of the reconstruction threshold $\lambda_R(k)$ for large $k$.

The Kesten-Stigum bound is known to be the correct bound for robust reconstruction for all Markov models~\cite{JM}.  Robust reconstruction asks whether reconstruction is possible after adding a large amount of noise to the spins in level $n$.  It was shown in~\cite{JM} that when $\theta^2(k-1)<1$ after adding enough noise to the spins at level $n$, the ``information'' provided by the modified spins at level $n$ decays exponentially quickly.

In both the colouring model and the hardcore model the reconstruction
threshold is far from the Kesten-Stigum bound for large $k$.  In the
case of the hardcore model $\theta^2 (k-1)= (1+o(1)\frac1k \ln^2 k$.
As such, given a noisy version of the spins at level $n$, the
information on the root decays rapidly as $n$ grows.  In the colouring
model close to optimal bounds~\cite{BVVW,Sly:08} were obtained by
first showing that, when $n$ is small, the information on the root is
sufficiently small.
Then  a quantitative version
of~\cite{JM} establishes that the information on the root converges to
0 exponentially quickly.  The hardcore model behaves similarly.
Indeed, the form of our bound in equation~(\ref{e:omegaBound}) is
strikingly similar to the bound for the $q$-coloring model which
states that reconstruction (resp. non-reconstruction) holds when the
degree is at least (resp. at most) $q[\ln q + \ln \ln q + O(1)]$.

Our proof then proceeds as follows.  We first establish that when
$\omega$ satisfies (\ref{e:omegaBound}) then even for a tree of depth
3 there is already significant loss of information of the spin at the
root.  In particular we show that if the state of the root is 1 then the typical
posterior probability that the state of the root is 1 given the spins at level 3
will be less than $\frac12$. The result is completed by linearizing
the standard tree recursion as in~\cite{BCMR,Sly:09}.
In this part of the proof we closely follow the notation of~\cite{BCMR} who analyzed the reconstruction problem for the Ising model with small external field. We do not require the full strength of their analysis as in our case we are far from the Kesten-Stigum bound.
We show that a quantity which we refer to as the \emph{magnetization} decays
exponentially fast to~0.    The magnetization provides a bound on the
posterior probabilities and this completes the result.

\subsubsection*{Replica Symmetry Breaking and Finding Large Independent Sets}

The reconstruction problem plays a deep role in the geometry of the
space of solutions of random constraint satisfaction problems.  While
for problems with few constraints the space of solutions is connected
and finding solutions is generally easy, as the number of constraints
increases the space may break into exponentially many small clusters.
Physicists, using powerful but non-rigorous ``replica symmetry
breaking'' heuristics, predicted that the clustering phase transition
exactly coincides with the reconstruction region on the associated
tree model~\cite{MM,KMRSZ}.
This picture was rigorously established (up to first order terms) for the colouring and
satisfiability problems~\cite{AchCA} and further extended to sparse random graphs by~\cite{Montanari_Restrepo_T}.
As solutions are far apart,
local search algorithms will in general fail. Indeed for both the
colouring and SAT models, no algorithm is known to find solutions in
the clustered phase.  It has been conjectured to be computationally
intractable beyond this phase transition~\cite{AchCA}.
\linebreak

The associated CSP for the hardcore model  corresponds to finding
large independent sets in random $k$-regular graphs. The replica
heuristics again predict that the space of large independent sets
should be clustered in the reconstruction regime.  Specifically this
refers to independent sets of size $s n$ where $s > \pi_1(R)$, the
density of 1's in the hardcore model at the reconstruction threshold.
It is known that the largest independent set is with high probability
$\frac{(2-o(1))\ln k}{k} n$~\cite{CFRR}. On the other hand the best
known algorithm finds independent sets only of size $\frac{(1+o(1))\ln
  k}{k}n$ which is equal to $\pi_1(R) n$~\cite{Wormald}.  This is
consistent with the physics predictions and it would be of interest to
determine if the space of independent sets indeed exhibits the same
clustering phenomena as colourings and SAT at the reconstruction
threshold.  Determining the reconstruction threshold more precisely
thus has implications for the problem of finding large independent
sets in random graphs.

\subsubsection*{Glauber Dynamics on trees}

The reconstruction threshold plays a key role in the study of the rate of convergence of the Glauber dynamics markov chain for sampling spin systems on trees.  This problem has received considerable attention~(see e.g. \cite{BKMP,DLP,MSW,MSW2,TVVY}) and in the case of the Ising model, the mixing time is known to undergo a phase transition from $\theta(n\ln n)$ in the non-reconstruction regime to $n^{1+\theta(1)}$ in the reconstruction regime~\cite{BKMP}.  In fact, the mixing time is $n^{1+\theta(1)}$ for any spin system above the reconstruction threshold.  A similar transition was shown to take place for the colouring model~\cite{TVVY}.  Sharp bounds of this type are not known from the hardcore model, however, it is predicted that the Glauber dynamics should again be $O(n \log n)$ in the non-reconstruction regime.

\section{Proof of Theorem~\ref{t:main}}

It is simple to show that non-reconstruction on the $k$-regular tree
is equivalent to non-reconstruction on the $(k-1)$-regular tree.  For
ease of notation we establish our bounds for the $k$-ary tree noting
that in equation~(\ref{e:omegaBound}) we have that $\ob(k+1)-\ob(k) =
o(k)$ so the difference can be absorbed in the error term.  Let
$\mathcal T$ denote the infinite $k$-ary tree  and let $T_n$ denote
the restriction of $\mathcal T$ to its first $n$ levels.

Before reading further, it might help the reader to quickly recall the notation from the end of Section~\ref{sec:hardcore_model}.
As in~\cite{BCMR} we analyse a random variable $X$ which denotes {\em
  weighted magnetization of the root} which is a function of the leaf
states of the tree.    We define $X=X(n)$ on $T_n$ by
\begin{eqnarray}\nonumber
X & = &\pi_0^{-1}[\pi_0\p(\sigma_\rho=1|A) - \pi_1\p(\sigma_\rho=0|A)] \\
\label{eq:mag2} & = &
\frac{1}{\pi_{01}}\left[\frac{\p[\sigma_\rho=1|A]}{\pi_1}-1\right]
\end{eqnarray}
Since $\E_T[\p[\sigma_{\rho}=1|A]] = \p[\sigma_\rho=1]=\pi_1$, from
the above expression, we
have that $\E[X]=0$. Also, $X \leq 1$ since $\p[\sigma_{\rho}=1|A] \leq
1$. We will make extensive use of the following second moments of the
magnetization.
\begin{eqnarray}\nonumber
\overline X = \E_T[X^2], \ \ \ \overline X_1 = \E_T^1[X^2], \ \ \
\overline X_0 = \E_T^0[X^2]
\end{eqnarray}
With these definitions in hand, by the definition in (\ref{eq:non-rec}) we can
characterize non-reconstruction as follows.
\begin{proposition}\label{p:bar-x-non-reconst} Non-reconstruction for
  the model $(\mathcal T,M)$ is equivalent
  to $$\lim_{n \rightarrow \infty} \overline X(n) = 0,$$
where $\overline X(n) = \E_{T_n}[X^2]$.
\end{proposition}

In the remainder of the proof we derive bounds for $\overline
X$. We begin by showing that already for a 3 level tree, $\overline X$
becomes small.  Then we establish a recurrence along the lines of
\cite{BCMR} that shows that once $\overline X$ is sufficiently small,
it must converge to 0.  As this part of the derivation follows the
calculation in \cite{BCMR} we will adopt their notation in
places.  Non-reconstruction is then a consequence of
Proposition~\ref{p:bar-x-non-reconst}.  In the next lemma we
determine some basic properties of $X$.

\begin{lemma}\label{lem:basic-relations} The following relations hold:
\begin{enumerate}[a)]
\item $\E_T[X] = \pi_1\E^1_T[X] + \pi_0\E^0_T[X] = 0.$
\item $\overline X = \pi_1\overline X_1 + \pi_0 \overline X_0.$
\item $\E_T^1[X] = \pi_{01} \overline X$ and $\E_T^0[X] = - \overline X.$
\end{enumerate}
\end{lemma}

\begin{proof}
Note that for any random variable which depends only on the states at
the leaves, $f=f(A)$, we have $E_T[f] = \pi_1\E^1_T[f] +
\pi_0\E^0_T[f]$. Parts $a)$ and $b)$ therefore follow since $X$ is a random
variable that is a function of the states at the leaves.
For part $c)$ we proceed as follows. The first and last equalities
below follow from (\ref{eq:mag2}).
\begin{eqnarray}
\nonumber \E_T^1[X] & = &
\pi_{01}^{-1}\displaystyle\sum_{A}\p_T[\sigma_{L}=A|\sigma_\rho=1]\left(
\frac{\p_T[\sigma_\rho=1|A]}{\pi_1}-1 \right) \\
\nonumber & = &
\pi_{01}^{-1} \displaystyle\sum_{A}\p_T[\sigma_{L}=A]\frac{\p_T[\sigma_\rho=1|A]}{\pi_1}\left(
\frac{\p_T[\sigma_\rho=1|A]}{\pi_1} -1 \right) \\
\nonumber & = &
\pi_{01}^{-1}\left( \frac{\E_T[(\p_T[\sigma_\rho=1|A])^2]}{\pi_1^2} -1
\right) \\
\nonumber & = & \pi_{01}\E[X^2]
\end{eqnarray}
The second part of $c)$ follows by combining this with $a)$. \qed
\end{proof}

The following proposition estimates typical posterior probabilities
which we will use to bound $\overline X$.  For a finite tree $T$ let
$T^i$ be the subtrees rooted at the children of the root $u_i$.

\begin{proposition}\label{prop:expected-small}
For a finite tree $T$ we have that
\begin{enumerate}[a)]
\item For any configuration at the leaves
  $A=(A_1,\cdots,A_k)$,
\begin{eqnarray}
\nonumber \p_T[\sigma_\rho=0|\sigma_L=A]=\Bigl({1+\lambda
  \prod_{i}\p_{T^i}[\sigma_{u_i}=0 | \sigma_{L_i} = A_i]}\Bigr)^{-1}.
\end{eqnarray}

\item Let $\mathcal A$ be the set of leaf
  configurations
\begin{eqnarray}
\nonumber \mathcal A = \left\{\sigma(L) \ | \ \p[\sigma_\rho=0|\sigma(L)] =
\frac{1}{2}\left(1+\frac{1}{1+2\lambda}\right) \right\}.
\end{eqnarray}
Then
\begin{eqnarray}
\nonumber \frac{\p^0_T[\sigma(L) \in \mathcal A]}{\p^1_T[\sigma(L) \in
    \mathcal A]} =
\frac{\pi_1}{\pi_0}\frac{1+\lambda}{\lambda}.
\end{eqnarray}

\item Let $\beta>\ln 2 - \ln \ln 2$ and $\w = \frac1k\big[\ln k + \ln
  \ln k - \ln
  \ln \ln k -\beta \big]$. Then in the~3 level $k$-ary tree $T_3$ we have that
  $$\E^1_{T_3}[\p[\sigma_\rho=1 | \sigma(L)]] \leq \frac{1}{2}.$$

\end{enumerate}

\end{proposition}

\begin{proof} Part $a)$ is a consequence of standard tree recursions
  for Markov models established using Bayes rule.

For part $b)$ first note that
\begin{eqnarray}\label{e:calARelation}
\p[\sigma_\rho=1 \ | \ \sigma(L) \in \mathcal A] & = & 1-\p[\sigma_\rho=0
  \ | \ \sigma(L) \in \mathcal A] \nonumber\\
& = & \frac{1}{2}\left(1-\frac{1}{1+2\lambda}\right)
\end{eqnarray}
Now,
\begin{eqnarray}
\nonumber \p^0_T[\sigma(L) \in \mathcal A]  & = & \frac{\p[\sigma_\rho=0 \ |
    \ \sigma(L) \in \mathcal A] \p[ \sigma(L) \in \mathcal A]}{\pi_0}\\
\nonumber & = &  \frac{\pi_1}{\pi_0} \frac{1+\lambda}{\lambda}
\left(\frac{\p[\sigma_\rho=1 \ |
    \ \sigma(L) \in \mathcal A] \p[ \sigma(L) \in \mathcal
    A]}{\pi_1}\right)\\
\nonumber & = & \frac{\pi_1}{\pi_0}
\frac{1+\lambda}{\lambda}\p^1_T[\sigma(L) \in \mathcal A]
\end{eqnarray}
where the first and third equations follow by definition of
conditional probabilities and the second follows
from~(\ref{e:calARelation}) which establishes $b)$.

For part $c)$, we start  by calculating the probability of certain
posterior probabilities for  trees of small depth. Note that with our
assumption on $\omega$ we have that
\[
\lambda=\omega(1+\omega)^{k}=\frac{ e^{-\beta}  \ln^2 k}{ \ln \ln k}
\]
By part $a)$, since $\sigma(L)\equiv 1$ under $\p^1$ we have that
\[
\p^1_{T_1}[\sigma_\rho=0|\sigma(L)] =  \frac{1}{1+\lambda}  \ w. p. \ 1.
\]

Also,
\[
\p_{T_1}(u_i=0 \ \forall \ i | \sigma_\rho=0) =
\left(\frac{1}{1+\w}\right)^k 
\]

Using the two equations above, we have that

\[
\p^0_{T_1}(\sigma_\rho=0|\sigma(L)) = \left\{
\begin{array}{ll}
1 & \ w.p. \  1- \left(\frac{1}{1+\w}\right)^k \\
\frac{1}{1+\lambda} & \ w.p. \  \left(\frac{1}{1+\w}\right)^k
\end{array}\right.
\]

Applying part $a)$ to a tree of depth $2$, we have

\[
\p_{T_2}^1[\sigma_\rho=0 | \sigma(L)] = \frac{1}{1+\lambda
  \prod_{i}\p_{T_1}^0[\sigma_{u_i}=0 | \sigma(L)]}
\]

Therefore

\begin{eqnarray}\label{eq:p-t-2}
\p_{T_2}^1[\sigma_\rho=0 | \sigma(L)] = \left\{
\begin{array}{ll}
\frac{1}{1+\lambda} &
\ w.p.\ \left(1-(\frac{1}{1+\w})^k\right)^k
\\
\frac{1}{2}\left(1+\frac{1}{1+2\lambda}\right) &
\ w.p. \ \left(1-(\frac{1}{1+\w})^k\right)^{k-1}  \left(\frac{1}{1+\w}\right)^kk
\\
> \frac{1}{2}\left(1+\frac{1}{1+2\lambda}\right) & \ o.w.
\end{array}\right.
\end{eqnarray}

By part $b)$ with $\mathcal A$ as defined, and (\ref{eq:p-t-2}) we
have that after substituting the expressions for $\lambda$ and
$\omega$,
\begin{eqnarray}
\p_{T_2}^0[\sigma(L) \in \mathcal A] & = &
\frac{\pi_1}{\pi_0}\frac{1+\lambda}{\lambda} \p_{T_2}^1[\sigma(L) \in
  \mathcal A] \nonumber\\
& = & \frac{\w(1+\lambda)}{\lambda(1+\w)}
\left(1-\left(\frac{1}{1+w}\right)^k\right)^{k-1}
\left(\frac{1}{1+\w}\right)^k k\nonumber\\
\label{eq:bound-on-p} & \geq & (1-o_k(1))\frac{e^{\beta} \ln \ln k }{k}
\end{eqnarray}

We can now calculate the values of $P_{T_3}^1[\sigma_\rho=0 | \sigma(L)]$
as follows. By part $a)$

\[
\p_{T_3}^1[\sigma_\rho=0 | \sigma(L)]  = \frac{1}{1+\lambda
  \prod_{i}\p_{T_2}^0[\sigma_{u_i}=0 | \sigma(L)]}
\]

Denote
\[
p = \frac{\w(1+\lambda)}{\lambda(1+\w)}
\left(1-\left(\frac{1}{1+w}\right)^k\right)^{k-1}\left(\frac{1}{1+\w}\right)^k k
\]

By Chernoff bounds, and the bound on $p$ from
(\ref{eq:bound-on-p}), $$\p(Bin(k,p)< e^{\beta}\ln \ln k -
2\sqrt{e^{\beta}\ln \ln k})< \frac{1}{3}.$$ Finally, by the definition
of $\mathcal A$,

\[
\p_{T_2}^0[\sigma_{u_i}=0 | \sigma(L) \in \mathcal A] =
\frac{1}{2}\left(1+\frac{1}{1+2\lambda}\right)
\]

and hence,

\[
\E^1_{T_3}[\p[\sigma_\rho=1 | \sigma(L)]] \leq \left(1- \frac{1}{1+\lambda
  [2(1-o_k(1))]^{-(e^\beta \ln \ln k - 2\sqrt{e^\beta \ln \ln k)}}} \right)\frac{2}{3}
+\ \frac{1}{3}
\]
By taking $k$ large enough above, we conclude that for $\beta$ and
large enough $k$,
\[
\E^1_{T_3}[\p[\sigma_\rho=1 | \sigma(L)]] \leq \frac{1}{2}
\]
\qed
\end{proof}

\begin{lemma}\label{lem:finite-levels}Let $\beta>\ln 2 - \ln \ln 2$
  and $\w = \frac1k\big[\ln k + \ln \ln k - \ln
  \ln \ln k -\beta \big]$.  For $k$ large enough,
$$\overline X(3) \leq \frac{\w}{2}.$$
\end{lemma}

\begin{proof}
By part $c)$ of Lemma \ref{lem:basic-relations}, and part $c)$ of
Proposition \ref{prop:expected-small},
\begin{eqnarray*}
\overline X(3) & = & \frac{1}{\pi_{01}^{2}} \left(
\frac{\E^1_{T_3}[\p[\sigma_\rho=1 \ | \ \sigma(L)]] }{\pi_1}- 1 \right) \\
& \leq & \frac{1}{\pi_{01}^{2}} \left(
\frac{1}{2\pi_1}- 1 \right)\\
& \leq  & \frac{\w}{2}
\end{eqnarray*}
\end{proof}

Next, we present a recursion for $\overline X$ and complete the proof of
the main result. The developement of the recursion follows the steps
in \cite{BCMR} closely so we follow their notation and omit some of
the calculations in this short version.

\begin{figure}[t]
\center \includegraphics[width=7cm]{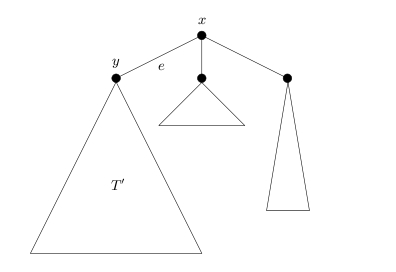}
\caption{A finite tree $T$}
\end{figure}

\subsubsection*{Magnetisation of a child} With $T$ and $x$ as defined
previously, let $y$ be a child of $x$ and let $T'$ be the subtree of $T$
rooted at $y$ (see Figure 1). Let $A'$ be the restriction of $A$ to
the leaves of $T'$. Let $Y=Y(A')$ denote the magnetization of $y$.

\begin{lemma} \label{lem:child-magnetization} We have
\begin{enumerate}[a)]
\item $\E^1_T[Y] = \theta \E^1_{T'}[Y]$ and $\E^0_T[Y] = \theta
  \E^0_{T'}[Y]$.
\item $\E^1_T[Y^2] = (1-\theta) \E_{T'}[Y^2]+ \theta \E_{T'}^1[Y^2].$
\item $\E^0_T[Y^2] = (1-\theta) \E_{T'}[Y^2]+ \theta \E_{T'}^0[Y^2].$
\end{enumerate}
\end{lemma}

The proof follows from the first part of Lemma \ref{lem:basic-relations} and the
Markov property when we condition on $x$.

Next, we can write the effect on the magnetization of adding an edge
to the root and merging roots of two trees as follows.
Referring to Figure 2, let $T'$ (resp. $T''$) be a finite tree
rooted at $y$ (resp. $z$) with the channel on all edges being given
$M$, leaf states $A$ (resp $A''$) and weighted magnetisation at the
root $Y$ (resp. $Z$). Now
add an edge $(\hat y, z)$ to $T''$ to obtain a new tree $\hat T$. Then
merge $\hat T$ with $T'$ by
identifying $y = \hat y$ to obtain a new tree $T$. To avoid
ambiguities, denote by $x$ the root of $T$ and $X$ the
magnetization of the root of $T$. We
let $A = (A', A'')$ be the leaf state of $T$. Let $\hat Y$ be the
magnetization of the root of $\hat T$.\\

{\bf Note:} In the above construction, the vertex $y$ is a vertex ``at
the same level'' as $x$, and not a child of $x$ as it was in Lemma
\ref{lem:child-magnetization}. \\

\begin{figure}[t]
\center \includegraphics[width=9cm]{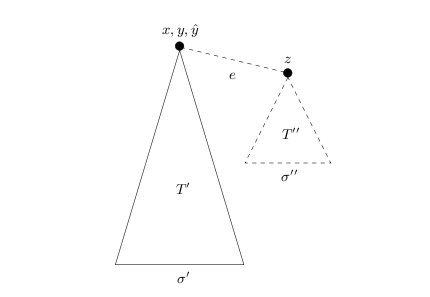}
\caption{The tree $T$ after obtained after merging $T'$ and $T''$. The
dashed subtree is $\hat T$.}
\end{figure}

\begin{lemma}\label{lem:add-edge}With the notation above, $\hat Y = \theta Z.$
\end{lemma}

The proof follows by applying Bayes rule, the Markov property and
Lemma \ref{lem:basic-relations}. These facts also imply that

\begin{lemma}\label{lem:merge}
For any tree $\hat{T}$,
\[
X  =  \frac{Y +\hat Y + \Delta Y \hat Y }{1+\pi_{01}Y \hat Y}.
\]
\end{lemma}

With these lemmas in hand we can use the following relation to derive
a recursive upper bound on the second moments. We will use the
expansion

\[
\frac{1}{1+r} = 1-r + r^2 \frac{1}{1+r}.
\]

Taking $r = \pi_{01}Y \hat Y$, by Lemma \ref{lem:merge} we have
\begin{eqnarray}\label{eq:rec-expansion}
\nonumber X & = & (Y +\hat Y + \Delta Y \hat Y ) \left[ 1-\pi_{01}Y \hat Y +
  (\pi_{01}Y \hat Y )^2 \frac{1}{1+\pi_{01}Y \hat Y} \right]\\
\nonumber & = & Y +\hat Y  + \Delta Y \hat Y - \pi_{01}Y \hat Y \left( Y +\hat Y
  + \Delta Y \hat Y \right) + (\pi_{01})^2(Y \hat Y)^2 X\\
& \le & Y +\hat Y  + \Delta Y \hat Y - \pi_{01}Y \hat Y \left( Y +\hat Y
  + \Delta Y \hat Y \right) + (\pi_{01})^2(Y \hat Y)^2
\end{eqnarray}
where the last inequality follows since $X \leq 1$ with probability 1.

Let $\rho' = \overline Y_1 / \overline Y $ and $\rho'' =
\overline Z_1 / \overline Z$. Below, the moments $\overline Y$
etc. are defined according to the appropriate measures over the tree
rooted at $y$ (i.e. $T'$) etc.

By applying Lemmas \ref{lem:basic-relations}, \ref{lem:child-magnetization} and
\ref{lem:add-edge}, we have the following relations.

\begin{eqnarray}
\E^1_T[X] = \pi_{01} \overline X, \ \ \ \E^1_T[Y] = \pi_{01} \overline
y, \ \ \ \E^1_T[Y^2] = \overline Y \rho' \nonumber \\
\label{eq:conditional-second-moment} \E^1_T[\hat Y] = \pi_{01}\theta^2
\overline Z, \ \ \ \E^1_T[\hat Y^2]
= \theta^2 \overline Z ((1-\theta) + \theta \rho'')
\end{eqnarray}


Applying $(\pi_{01})^{-1}E_T^1[\cdot]$ to both sides of
(\ref{eq:rec-expansion}), we obtain the following.
\begin{eqnarray}
\nonumber \overline X & \leq & \overline Y + \theta^2 \overline Z +
\Delta\pi_{01}\overline Y \overline Z
-\pi_{01}\theta^2 \overline Y \overline Z \rho'
-\pi_{01}\theta^2 \overline Y \overline Z ((1-\theta) + \theta \rho'')\\
&& -\Delta\theta^2 \overline Y \overline Z \rho' ((1-\theta) + \theta \rho'')
+ \pi_{01}\theta^2 \overline Y \overline Z \rho' ((1-\theta) + \theta
\rho'')\nonumber \\
& = & \overline Y + \theta^2 \overline Z -\pi_{01} \theta^2 \overline Y \overline Z
[\mathcal A - \Delta \mathcal B] \nonumber
\end{eqnarray}
where
\begin{eqnarray}
\mathcal A & = & \rho' + (1-\rho')[(1-\theta) + \theta\rho''],
\nonumber \\
\mathrm{and} \ \
\mathcal B & = & 1- (\pi_{01})^{-1}\rho'[(1-\theta) + \theta\rho''] = 1-
\frac{\w}{1+\w} \rho'[(1-\theta) + \theta\rho''] . \nonumber
\end{eqnarray}

If $\mathcal A - \Delta \mathcal B \geq 0$, this would already give a
sufficiently good recursion to show that $\overline X(n)$ goes to $0$,
so we will assume is negative and try to get a good (negative) lower bound.
First note that by their definition $\rho',\rho'' \geq 0$. Further
since $\overline Y = \pi_1 \overline Y_1 + \pi_0 \overline Y_0$, \[
\rho' \leq (\pi_1)^{-1} = \frac{1+2\w}{\w}.\] Similarly, \[ \rho''
\leq (\pi_1)^{-1} = \frac{1+2\w}{\w}.\]

Since $\E_T^1[\hat Y^2]$ and $\overline Z \geq 0$, it follows from
(\ref{eq:conditional-second-moment}) that $(1-\theta) + \theta\rho'' \geq
0$. Together with the fact that $\rho' \geq 0$, this implies that
$\mathcal B \leq 1$.

Since $\mathcal A$ is multi-linear in $(\rho',\rho'')$, to minimize it,
its sufficient to consider the extreme cases. When $\rho' = 0$,
$\mathcal A$ is minimized at the upper bound of $\rho''$ and hence
\[
\mathcal A \geq 1-\pi_{01}\frac{\w}{1+\w} = 0.
\]

When $\rho' = (\pi_1)^{-1}$,
\[
\mathcal A = (\pi_1)^{-1} + (1-(\pi_1)^{-1})[1-\theta(1-\rho'')] \geq 0.
\]

Hence, we have
\[
\overline X \leq \overline Y + \theta^2 \overline Z + \frac{1}{1+\w}\overline Y \overline Z.
\]

Applying this recursively to the tree, we obtain the following
recursion for the moments.

\[
\overline X \leq (1+\w)\theta^2 \left[\left(1+\frac{\overline
  Z}{1+\w}\right)^k-1\right]
\]

We bound the $(1+x)^k-1$ term as,
\[
|(1+x)^k -1| \leq e^{|x|k}-1 =\int_0^{|x|k} e^s \ ds \leq e^{|x|k}k|x|
\]

and this implies the following recursion.

\begin{theorem}\label{thm:recursion}
If for some $n$, $\overline X(n) \leq \frac{\w}{2}$, we have that
\[
\overline X(n+1) \leq \w^2e^{\frac{1}{2}\w k}k \overline X(n).
\]
\end{theorem}

Thus if $ \w^2e^{\frac{1}{2}\w k}k < 1$ then it follows from the recursion that
\begin{equation}\label{e:recursionLimit}
\lim_n \overline X(n) = 0.
\end{equation}
When $\w = \frac1k\big[\ln k + \ln \ln k - \ln \ln \ln k
  -\beta\big]$ and $\beta> \ln 2 - \ln \ln 2$, by Lemma
\ref{lem:finite-levels}, for $k$ large enough, $\overline X(3) \leq
\frac{\w}{2}$. Hence by equation~(\ref{e:recursionLimit}) we have that
$\overline X(n) \to 0$ and so by Proposition~\ref{p:bar-x-non-reconst}
we have non-reconstruction. Since reconstruction is monotone in
$\lambda$ and hence in $\omega$ it follows that we have
non-reconstruction for $\omega \leq \ob$ for large $k$.  This
completes the proof of Theorem~\ref{t:main}.

\newpage

\end{document}